\documentclass[a4paper]{amsart}

\usepackage{amsthm}
\usepackage{amsmath}
\usepackage{amssymb}
\usepackage{latexsym}
\usepackage{enumerate}
\usepackage{graphicx}
\usepackage[utf8]{inputenc}

\usepackage{epsfig} 
\usepackage{pgf}
\usepackage{tikz}
\usepackage{float}
\usepackage{dsfont}
\usepackage{times}

\renewcommand{\thetheoremName}

\newtheorem{proposition[[]]}[theoremName]{Proposition G}

\newtheorem{theorem}{Theorem}[section]
\newtheorem{lemma}[theorem]{Lemma}

\newtheorem{proposition}[theorem]{Proposition}
\newtheorem{corollary}[theorem]{Corollary}

\theoremstyle{definition}
\newtheorem{definition}[theorem]{Definition}

\numberwithin{equation}{section}

\newcommand{\Hess}{\operatorname{Hess}}

\newcommand{\kan}{\mathbb{M}^{n}(\kappa)}

\newcommand{\erre}{\mathbb{R}}
\newcommand{\R}{\mathbb{R}}

\newcommand{\grad}{\operatorname{\nabla}}
\newcommand{\hess}{\textrm{Hess}}

\newcommand{\E}{\mathcal{E}}
\providecommand{\rad}{\mathop{\rm rad}\nolimits}
\newcommand{\Rho}{{\rho_{\erre^n}}}

\begin{document}

  \title[Total curvature and area growth of  tamed surfaces]{On the total curvature and extrinsic area growth of surfaces with tamed second fundamental form}

\author{Cristiane M. Brand\~ao}      
\address{Departamento de Matem\'atica universidade do Cear\'a-UFC, 60455-760 Fortaleza, CE, Brazil                       }
\email{crismbrandao@yahoo.com.br}


\author{Vicent Gimeno}      
\address{Department of Mathematics-INIT-IMAC, Universitat Jaume I, Castell\'o de la Plana, Spain                        
}
\email{gimenov@uji.es}

\thanks{Work partially supported by DGI grant MTM2010-21206-C02-02.}


\begin{abstract}
In this paper we show that a complete and non-compact surface immersed in the Euclidean space with quadratic extrinsic area growth has finite total curvature provided the surface has tamed second fundamental form and admits total curvature. In such a case we obtain  as well a generalized Chern-Osserman inequality. In the particular case of a surface of nonnegative curvature, we prove that the surface is diffeomorphic to the Euclidean plane if the surface has tamed second fundamental form, and that the surface is isometric to the Euclidean plane if the surface has strongly tamed second fundamental form.  In the last part of the paper  we characterize the fundamental tone of any submanifold of tamed second fundamental form immersed in an ambient space with a pole and quadratic decay of the radial sectional curvatures.  

\keywords{Total curvature \and Chern-Osserman inequality \and Fundamental tone \and area growth}
 \subjclass{35P15}
\end{abstract}

\maketitle

\section{Introduction}
Let $M$ be a complete non-compact surface, the total curvature of $M$ is the improper integral $\int_M K\, dA$ of the Gaussian curvature $K$ with respect to the volume element $dA$ of $M$. It is said that $M$ admits total curvature if for any compact exhaustion $\{\Omega_i\}$ of $M$, the limit
$$
\int_M K\,dA=\lim_{i\to\infty}\int_{\Omega_i}K\,dA 
$$
exists. Cohn-Vossen proved in \cite{CV} that $\int_M K\, dA\leq \chi(M)$, where $\chi(M)$ is the Euler characteristic of $M$. A well known theorem due to Huber \cite{Hub} states that if the negative part of the curvature $K_{-}=\max\{-K,0\}$ has finite integral, namely,
\begin{equation}\label{finite-negative}
\int_MK_{-}\,dA<\infty,
\end{equation}then, $\int_M K\, dA\leq \chi(M)$ and $M$ is conformally equivalent to a compact Riemann surface with finitely many punctures. Hartman, under the assumption (\ref{finite-negative}) proved in \cite{Har} that the area ${\rm A}(B_r)$ of a geodesic ball of radius $r$ at a fixed point must grow at most quadratically in $r$. Reciprocally, Li proved in \cite{Li97} that if $M$ has at most quadratic area growth, finite topology and the Gaussian curvature of $M$ is either non-positive or non-negative, near infinity of each end, then $M$ must have finite total curvature.

From an extrinsic point of view, in the setting of a minimal surface $M$ immersed in the Euclidean space $\erre^n$, it is  well known, see \cite{Choss, JM,Oss1,Oss}, that  if $M$ has finite total curvature then $M$ has finite topological type and  \emph{quadratic extrinsic area growth}, \emph{i.e.}, there exists a constant $C$ such that for any $r\in \R_+$
\begin{equation}
\text{Area}(M\cap B_r(0))\leq C r^2,
\end{equation}
where $B_r(0)$ denotes the geodesic ball centered at the origin $0\in \erre^n$ of radius $r$.

 Conversely, Q. Chen \cite{Ch1}, proved that if $M$ is an oriented complete  minimal surface in the Euclidean space $\erre^n$ with quadratic extrinsic area growth and finite topological type then $M$ has finite total curvature.

A natural question  is whether  an equivalent result relating the extrinsic area growth and the total curvature holds for a boarder class of complete surfaces in the Euclidean space. The aim of this paper is to provide an answer to this question under certain control of the second fundamental form of the immersion. 
 A surface $M$ is said to have tamed second fundamental form if for a (any) compact exhaustion $\{\Omega_i\}$ of $M$,
\begin{equation}\label{eq1.3}
a(M):=\lim_{i\to\infty}\left(\sup_{x\in M\setminus \Omega_i}\left\{\rho_M(x)\Vert\alpha(x)\Vert\right\}\right)<1,
\end{equation}
where $\rho_{M} (x) = {\rm dist}_{M}(x_0, x)$ is the  distance function on $M$ to a fixed point $x_0$, and $\Vert \alpha(x)\Vert $ is the  norm of the second fundamental form at $\varphi(x)$.  The notion of immersion with tamed second fundamental form was introduced in \cite{Pac} for submanifolds of $\mathbb{R}^{n}$ and in \cite{Pac2} for submanifolds of Hadamard manifolds. This notion can  be naturally extended to manifolds with a pole and radial sectional curvature bounded above, see \cite{GPGap}. Under the hypothesis of tamed second fundamental form and quadratic extrinsic area growth we can state the following result.

\begin{theorem}\label{perimeter-th}Let $M$ be an immersed complete oriented  surface of $\mathbb{R}^{n}$ with curvature function $K$ and tamed second fundamental form. Suppose that $M$  admits total curvature. Then, $M$ has finite total curvature ($\int_MKdA>-\infty$), if and only if,
$M$ has quadratic extrinsic area growth, \emph{i.e.}, there exists a constant $C_1$ such that,
\begin{equation}\label{upper-theo}
\text{A}(M\cap B_r(0))\leq C_1 r^2,
\end{equation}
for any $r$ large enough. Furthermore, if (\ref{upper-theo}) holds, then there exists a constant $C_0>0$ such that
\begin{equation}
\text{A}(M\cap B_r(0))\geq C_0 r^2,
\end{equation}
for any $r$ large enough.
\end{theorem}
Observe that the assumption that the surface admits total curvature (finite or infinite) can be achieved if the surface has semidefinite curvature (either nonpositive or nonnegative). As observed by Jorge-Meeks \cite{JM}, any complete $m$-dimensional submanifold $M$ of $ \mathbb{R}^{n}$ homeomorphic to a compact Riemannian manifold
$\overline{M}$ punctured at finite number of points
$\{p_{1},\ldots,p_{r}\}$ and having a well defined normal vector at
infinity have $a(M)=0$. In particular, any  complete minimal surfaces of
$\mathbb{R}^{n}$ with finite total curvature has tamed second fundamental form with $a(M)=0$. Anderson \cite{A1},  showed that a complete $m$-dimensional minimally
immersed submanifold $M$ of $\mathbb{R}^{n}$ has finite total scalar
curvature, $\smallint_{M}\Vert \alpha \Vert^{m} dV<\infty $, if and
only if $M$ is $C^{\infty}$- diffeomorphic to a compact 
Riemannian manifold $\overline{M}$ punctured at a finite number of
points $ \{p_{1},\ldots,p_{r}\}$
 and the Gauss map $\Phi$ on $M$
extends to a $C^{\infty}$- map $\overline{\Phi}$ on $\overline{M}$,
where $\Vert \alpha \Vert$ is the norm of the second fundamental
form of $M$.
In \cite{Pac}, Bessa, Jorge and Montenegro  showed that some aspects of Anderson's result hold on  complete immersed submanifolds  of $\mathbb{R}^{n}$ with tamed second fundamental form, i.e. they are properly immersed and  have finite topology,  meaning that $M$ is $C^{\infty}$-diffeomorphic to a  compact smooth manifold $\overline{M}$ with boundary. This result was extend by  Bessa-Costa to isometric immersions with tamed second fundamental form into Hadamard manifolds, \cite{Pac2} and by Gimeno-Palmer in \cite{GPGap} to  isometric immersion with tamed second fundamental form into  ambient manifolds with a pole and bounded radial sectional curvatures. They  also have shown that the volume growth and the number of ends of  submanifolds of dimension greater than $2$ are controlled with an appropriate decay of the extrinsic curvature.  

Assuming finite total curvature we can treat the two dimensional case obtaining the following Chern-Osserman type inequality. 
\begin{theorem}\label{theo2}Let $M$ be an oriented surface immersed in $\erre^n$ with curvature function $K$ and  tamed second fundamental form. Suppose in addition that $M$ has finite total curvature. 
Then, 
\begin{equation}
\left(1-a(M)^2\right)\widetilde C_0\leq 2\pi\chi(M)-\int_MKd\text{A}\leq C_1,
\end{equation}
where $\chi(M)$  is the Euler characteristic of $M$ and $\widetilde C_0,C_1$ are positive constants such that
\begin{equation}
\begin{aligned}
\text{A}(M\cap B_r(0))\leq & C_1 r^2,\\
\text{L}(M\cap S_r(0))\geq & \widetilde C_0 r,
\end{aligned}
\end{equation}
for any $r$ large enough.
\end{theorem} 

Cohn-Vossen proved in \cite{CV} that any complete and non-compact surface with nonnegative Gaussian curvature is diffeomorphic to $\erre^2$, or if not, it is flat. By using the above theorem  \ref{theo2} we can therefore state the following corollary 

\begin{corollary}\label{nonnegative}Let $M$ be an oriented surface immersed in $\erre^n$ with nonnegative curvature function ($K\geq 0$) and  tamed second fundamental form. Then, $M$ is diffeomorphic to $\erre^2$. Moreover, in the particular case when $M$ is flat, $M$ is isometric to $\erre^2$.
\end{corollary} 
  
Petrunin and Tuschmann in \cite{Petrunin2001} , solving a conjecture of Gromov \cite{Ballmann85} (see also \cite{Drees94}), proved that if a complete simply connected manifold $M^n$ of dimension greater than $2$ ($n\geq 3$) has nonnegative sectional curvatures $K\geq 0$  and is \emph{asymptotically flat} then $M$ is isometric to $\erre^n$. Here, asymptotically flat means that
\begin{equation}
k(t)t^{2}\to 0\quad (t\to\infty),
\end{equation}
where $k(t)$ is the supremum of $\vert K\vert$ on $M\setminus B_t(o)$ for some fixed point $o\in M$, ($B_t(o)$ being the geodesic ball of $M$ of radius $t$ centered at $o$). We can extend this intrinsic result to dimension $2$ but using an extrinsic approach. we will say that a surface immersed in $\erre^n$ has \emph{strongly} tamed second fundamental form if for some $\epsilon$, and for a (any) compact exhaustion $\{\Omega_i\}$ of $M$,
\begin{equation}
\lim_{i\to\infty}\sup_{x\in M\setminus \Omega_i}\left\{\rho_M(x)^{1+\epsilon}\Vert\alpha(x)\Vert\right\}<1.
\end{equation}
In such a case we obtain the following corollary
\begin{corollary}\label{cor1.4}
Let $M$ be an oriented surface immersed in $\erre^n$ with nonnegative curvature function ($K\geq 0$) and  strongly tamed second fundamental form. Then,  $M$ is isometric to $\erre^2$.
\end{corollary}

In proposition \ref{prop3.2} we will show that for a surface $M$ immersed in the Euclidean space $\erre^n$ with tamed second fundamental form, a necessary and sufficient condition to attain  quadratic extrinsic area growth is to have \emph{linear extrinsic perimeter growth}. Namely, there exists a constant $\widetilde C$, such that for any $r$ large enough,
$$
{\rm L}(M\cap S_r(0))\leq \widetilde C r
$$
where here $S_r(0)$ stands for the geodesic sphere of radius $r$ centered at $0\in \erre^n$.

It is also interesting to study the fundamental tone of submanifolds with tamed second fundamental form. Recall that the fundamental tone $\lambda^{\ast}(M)$ of a complete and non-compact Riemannian manifold $M$ is given by\[\lambda^{\ast}(M)=\inf\left\{\frac{\int_{M}\vert \nabla u \vert^{2}d\mu}{\int_{M}u^2d\mu},\,  u \in C^{\infty}_{0}(M)\setminus \{0\} \right\}\]It is well known that complete surfaces with finite total curvature are parabolic, see \cite{I1} and the proof of \cite[theorem 12.2]{Markv2003}. Taking into account that surfaces with positive fundamental tone  are hyperbolic surfaces, see \cite{GriExp}, one  concludes that surfaces with finite total curvature has zero fundamental tone as well as the surfaces with tamed second fundamental form with  quadratic extrinsic area growth.  

Observe that tamed second fundamental form implies certain quadratic decay of the Gaussian curvature. Actually, for the fundamental tone of submanifolds with tamed second fundamental form in an ambient space with a pole, we can state something more general using quadratic decay of the curvature again but in a completely different approach.

Let $N$ be a Riemannian manifold with a pole $p$ and radial sectional curvature bounded below \[K^{\rm rad}_{N}(x)\geq  B(\rho_{N}(x))\]   along the rays issuing from $p$, where $B\in C^{\infty}([0,\infty))$. The behavior of this comparison function  imposes restrictions on  the fundamental tone of tamed immersions as we can show in the following theorem.

\begin{theorem}\label{tone-theo}Let $\varphi\colon M \hookrightarrow N$ be an isometric immersion of a $m$-dimensional complete Riemannian manifold into an $n$-dimensional ambient manifold $N$ which possesses a pole and radial sectional curvatures bounded from below and above by
\begin{equation}
 B(\rho_{N}(x))\leq K^{\rm rad}_{N}(x) \leq 0.
\end{equation}
With $B\in C^\infty[0,\infty)$ such that for any $t> 0$,
\begin{equation}
B(t)\geq \frac{-2}{t^2}.
\end{equation}
Suppose moreover that the norm of the second fundamental form of the immersion is tamed. Namely, inequality (\ref{eq1.3}) holds. Then, $M$ has zero fundamental tone $\lambda^*(M)=0$.\end{theorem}Observe that the identity map ${\rm id}:N\to N$ induces an isometric immersion from $N$ to $N$ with vanishing second fundamental form, so with tamed second fundamental form. We can therefore state a purely intrinsic counterpart of theorem \ref{tone-theo}
\begin{corollary}Let $N$ be an $n$-dimensional Riemannian manifold $N$ which possesses a pole and radial sectional curvatures bounded from below and above by
\begin{equation}
 B(\rho_{N}(x))\leq K^{\rm rad}_{N}(x) \leq 0.
\end{equation}
With $B\in C^\infty[0,\infty)$ such that for any $t> 0$,
\begin{equation}
B(t)\geq \frac{-2}{t^2}.
\end{equation}Then, $N$ has zero fundamental tone, $\lambda^*(N)=0$.\end{corollary}

Much efford has been made in the understanding of Gap phenomenon for Riemannian manifolds. The classical results in this field (see for instance \cite{Siu-Yau, GreW2, Shiohama, GrePetZhu} and references therein) state that assuming certain kind of faster than quadratic decay of the curvature one obtains flatness and isometry to the Euclidean space. For example, by using theorem 1 of \cite{Shiohama},  any manifold $N^n$ with  a pole, dimension $n>2$, sectional curvatures $K_N$ bounded from below and above by $B(\rho_{N}(x))\leq K_{N}(x) \leq 0$, and with  bounding function $B\in C^{\infty}[0,\infty)$ satisfying 
\begin{equation}
\limsup_{t\to\infty}t^2B(t)=0,
\end{equation} 
is isometric to $\erre^n$.  Observe, moreover that in the most part of this paper is assumed only quadratic decay, or even slower than quadratic decay of the sectional curvature.
\section{Preliminaries}
Throughout this paper we shall study geometric and analytic properties of submanifolds immersed in an ambient Riemannian manifold with a pole. Recall that a Riemannian manifold $N$ is a Riemannian manifold with a pole if there exists a point $p\in N$ with empty cut locus, ${\rm cut}(p)=\emptyset$. In such a case the exponential map $\exp_p:T_pN\to N$ induces a diffeomorphism between $T_pN$ and $N$, and the distance function 
$$
\rho_N:N\to \erre,\quad x\to\rho_N(x)={\rm dist}_N(p,x),
$$
is a smooth function in $N\setminus\{p\}$. We will suppose moreover that the radial sectional curvatures of $N$, along the geodesics issuing from $p$, are bounded from above

\begin{equation}
\label{eqCurv}K_{N}(x)\leq - G(\rho_{N}(x))
\end{equation}
where $G\colon\mathbb{R}\rightarrow\mathbb{R}$ is a smooth
even function. Let $h$ be the solution of the following Cauchy problem
\begin{eqnarray}\label{eqh}
\left \{
\begin{array}{l}
h'' -Gh=0 \\
h(0)=0, h'(0)=1\\
\end{array} \right.
\end{eqnarray}
and let $I=[0,r_{0})\subseteq [0,\infty)$ be the maximal interval where $h$ is positive. If $G$ satisfies 
\begin{equation}\label{eqBMR-memoirs} t\int_{t}^{\infty}G_{-}(s)ds
\leq \frac{1}{4},
\end{equation} since it was shown that in this condition that
$h'\geq 0$, see \cite[Prop. 1.21]{bmr}, then $I=[0, +\infty)$. The Hessian Comparison Theorem states that \begin{equation}\label{eqBF6}
{\rm\hess}\, \rho_{N}(y)\geq \displaystyle\frac{h'}{h}(\rho_{N}(y))\{ \langle,\rangle - d\rho_{N}\otimes d\rho_{N}\}
\end{equation} in the sense of quadratic forms. If
\[
K_{N}(x)\geq -G(\rho_N(x))
\]
 then
\begin{equation}\label{eqBF6-b}
{\rm\hess}\, \rho_{N}(y)\leq \displaystyle\frac{h'}{h}(\rho_{N}(y))\{ \langle,\rangle - d\rho_{N}\otimes d\rho_{N}\}
\end{equation}See   \cite{PRS} and references therein.
 
 Associated to $N$, there are  $m$-dimensional model manifolds  $\mathbb{M}_{h}^{m} =  [0,\infty)  \times\mathbb{S}^{m-1}$ with  the metric $ds^{2}_{h}=dr^2+h^{2}\left(r\right)d\theta^{2} $, for every $m\geq 2$, where $h$ is the solution of \eqref{eqh}. Observe that these models have, 
radial sectional curvatures $-G\left(  r\right)  $. Let   $\varphi \colon M \hookrightarrow N$ be an isometric immersion  of a complete Riemannian $m$-manifold $M$ into  $N$. Let $x_0 \in M$ and let $\rho_{M} (x) = {\rm dist}_{M}(x_0, x)$ be the  distance function on $M$ to $x_0$.
Let $\{K_i\}_{i=0}^{\infty}$ be an exhaustion sequence of $M$ by nesting
compacts sets $K_i\subset K_{i+1}$ with $x_0 \in K_0$. Let $\{a_{i}(M)\}\subset [0,\infty]$ be a non-increasing
sequence of numbers defined  by
\[
\begin{array}{ccl}
a_i(M) = \sup \left \{ \displaystyle
\frac{h}{h'}(\rho_{M} (x))\cdot
\Vert \alpha (x)\Vert, \, x \in M \backslash K_i \right \},
\end{array}
\]
where $\Vert \alpha(x)\Vert $ is the  norm of the second fundamental form at $\varphi(x)$.
It is straightforward to show that the number  $a(M)=\displaystyle\lim a_i$ is independent on the sequence $\{K_{i}\}$ and on  $x_{0}$.

\begin{definition}
The immersion $\varphi$ has tamed second fundamental
form if $a(M)< 1$.
\end{definition}

Consider a smooth function $g:N\to \erre$ and the restriction $f=g\circ \varphi$. Identifying $X$ with $d\varphi(X)$ we have at $q\in M$ and for every $X\in T_qM$ that
\begin{equation}
\langle \nabla f,X\rangle=df(X)=dg(X)=\langle \nabla g,X\rangle.
\end{equation}
Hence we write
\begin{equation}
\nabla g=\nabla f+\nabla^{\perp}g,
\end{equation}
where $\nabla^{\perp}g$ is perpendicular to $T_qM$. In particular, for the extrinsic distance function $R=\rho_N\circ \varphi$
\begin{equation}
\nabla \rho_N=\nabla R+\nabla^\perp \rho_N.
\end{equation}

The following result, due to Gimeno-Palmer \cite{GPGap}, extends Bessa-Montenegro-Jorge \cite{Pac} and Bessa-Costa \cite{Pac2}. We shall present our proof of theorem \ref{logan} for the sake of completeness and to clarify the notation used in the paper.

\begin{theorem}[Gimeno-Palmer]\label{logan}Let $N$ be a Riemannian manifold with a pole  and radial sectional curvature $K_{N}^{\rad}(x)\leq -G(\rho_{N}(x))$, $G$ satisfying \eqref{eqBMR-memoirs}. If $\varphi\colon M \hookrightarrow N$ be an isometric immersion    of a complete Riemannian manifold with tamed second fundamental form then\begin{itemize}\item[i.] $\varphi$ is proper.\item[ii.] $M$ has finite topology.\end{itemize}
\end{theorem}

\begin{proof}
Since $a(M)< 1$ there exists $c\in (0,1)$ such that $a(M)< c<1$. Thus, there exists an
 $i\in \mathbb{N}$ such that $a(M)<a_i(M)<c$. This means that there
exists a  geodesic ball $B_{M}(r_0)\subset M$, with $K_i \subset
B_{M}(r_0)$, centered at $x_0$ with radius $r_0 > 0$ such that
\begin{equation}
\displaystyle \frac {h}{h'}(\rho_M
(x))\cdot\Vert \alpha(x)\Vert< c <1,\;\;\;
 {\rm for \;\; all}\;\; x \in M \backslash B_{M}(r_0).
\end{equation}
To fix the notation,  let $x_{0}\in M$, $p=\varphi (x_{0})$ and recall that $\rho_{M}(x)={\rm dist}_{M}(x_{0},x)$ and $\rho_{N}(y)={\rm dist}_{N}(p,y)$. Letting $\phi(t)=\int\limits_{0}^{t}h(s)ds$ define $f\colon M \to \mathbb{R}$ by $f=\phi\circ \rho_{N}\circ \varphi$.  It is straightforward to compute that for $X\in T_{x}M$, (identifying $d\varphi X=X$)
\begin{eqnarray}
\hess_{M} f(x)(X,X)&=& \hess_{N}\phi\circ \rho_{N}(\varphi(x))( X, X) + \langle{\rm\grad}\phi\circ\rho_{N},\alpha(X,X)\nonumber\\
&=& h'(\rho_{N})\langle{\rm\grad}\rho_{_N},X\rangle^2 +h(\rho_{N})
{\rm\hess}\rho_{N}(d\varphi X,d\varphi X)\\
&&  
+h(\rho_{N})\langle{\rm\grad}\rho_{N},\alpha(X,X)\rangle.\nonumber
\end{eqnarray}
By the Hessian Comparison theorem, we have that, 
\begin{equation}
\hess\rho_{N}(y)( X,X) \geq \displaystyle
 \frac{h'}{h}\left\{\Vert  X\Vert^2-\langle \grad \rho_{N},X\rangle^{2}\right\}
\end{equation} Therefore
for every $x \in M \backslash B_{M}(r_0)$ we have,
\begin{eqnarray}\label{eqf}
{\rm\hess} f(x)(X,X) &\geq&  h'(\rho_{N})\langle \grad \rho_{N}, X\rangle^2 +h(\rho_{N})[
\displaystyle \frac{h'}{h}(\rho_{N})\Vert X\Vert^2-\langle \grad \rho_{N},X\rangle^{2}]\nonumber \\ && -h(\rho_{N})\Vert \alpha\Vert\cdot\Vert X\Vert^2\nonumber \\
&=& h'(\rho_{N})\cdot\Vert X\Vert^2-h(\rho_{N})\cdot\Vert \alpha\Vert\cdot\Vert X\Vert^2 \\
&\geq& 
 h'(\rho_{N})\cdot(1-c)\cdot\Vert X\Vert^2.\nonumber
\end{eqnarray}

Let $\sigma : [0,\rho_{M}(x)] \rightarrow M$ be a minimal geodesic
joining $x_0$ to $x$. For all $t > r_0 $  we have  that $(f\circ
\sigma)''(t) = \hess f(\sigma(t))(\sigma' , \sigma')\geq h'(t)(1-c)$, where $h'(t)=h'(\rho_{N}(\varphi(\sigma(t))))$.

\noindent For $t \leq r_0$ we have that $(f\circ\sigma)''(t) \geq  b
= \inf\left\{{\rm\hess} f(x)(\nu,\nu),x \in B_{M}(r_0), \vert
\nu \vert = 1\right\}$. Hence
\begin{equation}
\begin{array}{ccl}
(f\circ \sigma)'(s) &=& (f\circ \sigma)'(0) +  \int_0^s (f \circ \sigma)''(\tau)d\tau\\
\\
&\geq& (f\circ \sigma)'(0)+  \int_0^{r_0} b \,d\tau + \int_{r_0}^s h'(\tau)(1 - c) d\tau\\
\\
&\geq& (f\circ \sigma)'(0)+ b\,r_0 + (1 - c)(h(s)-h(r_{0})).\\
\end{array}
\end{equation}
Now, since $\varphi (x_0)=p$, $\rho_{N}(\varphi(x_0))=0$ then
$(f\circ \sigma)'(0) = 0$,
 and $f(x_0)=0$, therefore
\begin{equation}
\begin{array}{ccl}
f(x)&=&  \int_0^{\rho_{M}(x)} (f\circ \sigma)'(s)ds\\
\\
&\geq&  \int_0^{\rho_{M}(x)}\left \{  b\,r_0 + (1 -c)(h(s) - h(r_0))\right\}ds\\
\\
&\geq&  b\,r_0 \,\rho_{M} (x) - (1-c)h(r_0)\rho_{M}(x) + (1-c)\int_0^{\rho_{M}(x)}h(s)ds\\
\\
&\geq& ( br_0 - (1-c)h(r_0))\rho_{M}(x) + (1-c)\int_0^{\rho_{M}(x)}h(s)ds\\
\end{array}
\end{equation}
Thus
\begin{equation}
 \phi(\rho_{N}( \varphi(x)) )\geq  ( br_0 - (1-c)h(r_0))\rho_{M}(x) + (1-c)\int_0^{\rho_{M}(x)}h(s)ds\label{eqP1}
\end{equation}
for all $x \in M$. Then we have that $\rho_{N}(\varphi)\to \infty $ when  $\rho_{M}(x)\to \infty$, and $\varphi$ is therefore proper.
Now let $ B_{N}(r_0)$ be the geodesic ball of $ N$ centered at $p$ with radius
$r_{0}$ and $S_{N}(r_0) =\partial B_{N}(r_0)$. Since  $\varphi$ is proper and  $a(M)<1$  we can take $r_0$ so that
\begin{equation}
\displaystyle \frac {h}{h'}(\rho_{M} (x))\Vert \alpha(x)\Vert \leq c <1,\;\;\;
{\rm for \;\; all}\;\; x \in M \backslash \varphi^{-1}(B_{N}(r_0))
\end{equation}
and by Sard's Theorem, see \cite[p.79]{DFN}, $r_{0}$ can be chosen so that
 $\Gamma_{r_0} = {\varphi(M)} { \displaystyle \cap} {S_{N}(r_0)}
\not = \emptyset $ is  a submanifold of ${\rm dim} \,\Gamma_{r_0}=
m-1$. For each $y\in \Gamma_{r_0}$, let us denote by
$T_y\Gamma_{r_0} \subset T_y\varphi(M)$ the tangent spaces of
$\Gamma_{r_0}$ and $\varphi(M)$ at $y$, respectively. Since  ${\rm dim}\, T_y\Gamma_{r_0} = m - 1$ and  ${\rm
dim}\,T_y\varphi(M)=m$, there exists only one unit vector $\nu(y) \in
T_y\varphi(M)$  such that \[ T_y\varphi(M) = T_y\Gamma_{r_0} \oplus
[[\nu(y)]], \] with $\langle \nu(y),\grad \!\rho_{N}(y)\rangle > 0$.
This defines a smooth vector field $\nu$ on a neighborhood $V$
of $\varphi^{-1}(\Gamma_{r_0})$. Here $[[\nu(y)]]$ is the
vector space generated by $\nu(y)$. Consider the function on
$\varphi(V)$ defined by
\begin{equation}\label{psi}
\begin{array}{ccl}
\psi(y)=\langle \nu,\grad \rho_{N}\rangle(y) = \langle
\nu,{\rm\grad}R \rangle (y)= \nu(y)(R),\,y=\varphi (x).
\end{array}
\end{equation}
Then $\psi(y) = 0$ if and only if every $x = \varphi^{-1}(y) \in V $
is a critical  point of the extrinsic distance function $R= \rho_{N} \circ \varphi$. Now for
each $y \in \Gamma_{r_0}$ fixed, let us consider the solution
$\xi(t,y) $  of the following Cauchy problem on $\varphi(M)$:
\begin{equation}\label{cauchy}
\left\{
\begin{array}{ccl}
\xi_t(t,y)&=&\displaystyle\frac{1}{\psi}\,\nu(\xi(t,y))\\
\\
\xi(0,y) &=& y\\
\end{array}
\right.
\end{equation}

We will prove that along  the integral curve $t \mapsto\xi(t,y)$
there are no critical points for $R$. For
this, consider the function $(\psi \circ \xi)(t,y)$ and observe that
\begin{equation}
\begin{array}{ccl}
\psi_t &=& \xi_t \langle {\rm\grad} \rho_{N} , \nu \rangle \\
\\
&=& \langle \overline{ \nabla}_{\xi_t} {\rm\grad} \rho_{N},\nu \rangle + \langle  {\rm\grad} \rho_{N}, \overline{ \nabla}_{\xi_t}\nu\rangle\\
\\
&=&\displaystyle\frac{1}{\psi} \langle \overline{ \nabla}_{\nu} {\rm\grad}
\rho_{N}, \nu  \rangle +
 \displaystyle \frac{1}{\psi}\langle  {\rm\grad} \rho_{N},\nabla_{\nu}\nu + \alpha(\nu,\nu)\rangle\\
\\
&=& \displaystyle\frac{1}{\psi} {\rm\hess} \rho_{N}(\nu,\nu) +
\frac{1}{\psi}\left[\langle {\rm\grad} \rho_{N},\nabla_{\nu}\nu \rangle + \langle {\rm\grad} \rho_{N}, \alpha(\nu,\nu)\rangle \right] \\
\\
&=& \displaystyle\frac{1}{\psi} \left[{\rm\hess} \rho_{N}(\nu,\nu) +
\langle\grad \rho_{N},\nabla_{\nu}\nu \rangle +
 \langle {\rm\grad}\rho_{N}, \alpha(\nu,\nu)\rangle \right].
\end{array}
\end{equation}
Thus
\begin{equation}\label{eqpsi}
\begin{array}{ccl}
\psi_t \psi &=& \hess \rho_{N}(\nu,\nu) + \langle\grad
\rho_{N},\nabla_{\nu}\nu \rangle + \langle\grad \rho_{N},
\alpha(\nu,\nu)\rangle
\end{array}
\end{equation}
Since  $ \langle \nu,\nu \rangle = 1$, we have at once that $\langle
\nabla_{\nu}\nu,\nu \rangle = 0$. As $\nabla_{\nu}\nu \in T_{x}M$,
we have that
$$
\langle{\rm\grad} \rho_{N},\nabla_{\nu}\nu \rangle
=\langle{\rm\grad} R,\nabla_{\nu}\nu \rangle.  $$ By equation
(\ref{psi}), we can write ${\rm \grad} R(x) =
\psi (\varphi(x))\cdot \nu(\varphi(x))$.  Since \[\grad R(x)\perp T_{\varphi
(x)}\Gamma_{\rho_{N} (y)},\] ($\Gamma_{\rho_{N} (y)}=\varphi(M)
\displaystyle \cap  S_{N}(\rho_{N} (y))$). Then
$$
\langle{\rm\grad} \rho_{N},\nabla_{\nu}\nu \rangle
=\langle{\rm\grad} R,\nabla_{\nu}\nu \rangle=
 \psi \langle\nu,\nabla_{\nu}\nu \rangle =0 .
$$
Writing
\begin{equation}\label{eqnu}
\nu(y) = \cos\beta(y)\; \grad \rho_{N} + \sin\beta(y)\; \omega
\end{equation}
and
\begin{equation}\label{eqro}
\grad \rho_{N} (y) = \cos \beta\; \nu(y) + \sin \beta\; \nu^*
\end{equation}
where $\langle \omega,\grad  \rho_{N} \rangle = 0$ and $\langle \nu,
\nu^* \rangle = 0$, the equation (\ref{eqpsi}) becomes
\begin{equation}
\begin{array}{ccl}
\psi_t \psi = \sin^2\beta\; \hess \rho_{N}(\omega,\omega) + \sin\beta \; \langle \nu^*, \alpha(\nu,\nu) \rangle. \\
\end{array}
\end{equation}
From (\ref{eqnu}) we have that $\psi(y)= \cos \beta(y)$
\begin{equation}
\psi_t \psi=\sqrt{1 - \psi^2}\sqrt{1 - \psi^2}\hess\rho_{N}(\omega,\omega) + \sqrt{1 - \psi^2}\langle \nu^*,\alpha(\nu,\nu)\rangle.\\
\end{equation}
Hence
\begin{equation}
\begin{array}{ccl}
\displaystyle\frac{\psi_t \psi}{\sqrt{1 - \psi^2}} &=& \sqrt{1 - \psi^2}\hess\rho_{N}(\omega,\omega) + \langle \nu^*,\alpha(\nu,\nu)\rangle.\\
\end{array}
\end{equation}
Thus we arrive at the following differential equation
\begin{equation}\label{eqdif1}
\begin{array}{ccl}
-(\sqrt{1 - \psi^2})_t &=&\sqrt{1 - \psi^2} \; \hess\rho_{N}(\omega,\omega) + \langle \nu^*,\alpha(\nu,\nu)\rangle\\
\end{array}
\end{equation}
The Hessian Comparison Theorem implies that
\begin{equation}
\hess\rho_{N}(\omega,\omega) \geq \displaystyle \frac{h'}{h}( \rho_{N}(\xi(t,y))).\\
\end{equation}
Substituting it in the equation $(\ref{eqdif1})$  obtain the
following inequality
\begin{equation}\label{eqdif2}
\begin{array}{ccl}
-(\sqrt{1 - \psi^2})_t &\geq& \sqrt{1 - \psi^2}\; \displaystyle \frac{h'}{h}(\rho_{N}(\xi(t,y)))\;+ \langle
\nu^*,\alpha(\nu,\nu)\rangle.\\
\end{array}
\end{equation}
Denoting  by $R(t,y)$  the restriction of $R=\rho_{N}\circ \varphi$ to $\varphi^{-1}(\xi(t,y))$ we have \[R(t,y)= R(\varphi^{-1}(\xi(t,y))) = \rho_{N}(\xi(t,y))\] On the other hand
we have that
\begin{equation}
R_t = \langle{\rm \grad} R,\frac{1}{\psi}\nu\rangle = \langle \psi
\nu, \frac{1}{\psi}\nu \rangle = 1
\end{equation}
then
\begin{equation}
R(t,y)= t + r_0.
\end{equation}
Writing  $\displaystyle \frac{h'}{h}(\rho_{N}(\xi(t,y))) = \displaystyle \frac{h'}{h}(t + r_0) $ in $(\ref{eqdif2})$ we have
\begin{equation}\label{eqdif3}
\begin{array}{ccl}
-(\sqrt{1 - \psi^2})_t \geq \sqrt{1 - \psi^2} \; \displaystyle\frac{h'}{h}(t + r_0) + \langle \nu^*,\alpha(\nu,\nu)\rangle\\

\end{array}
\end{equation}
Multiplying $(\ref{eqdif3})$ by $h(t + r_0)$, obtain
\begin{eqnarray*}
\begin{array}{ccl}
-\left[h(t + r_0)(\sqrt{1 - \psi^2})_t + h'(t + r_0)\sqrt{1 - \psi^2} \right]&\geq& h(t + r_0)\langle
\nu^*,\alpha(\nu,\nu)\rangle\\
\end{array}
\end{eqnarray*}
The last inequality  can be written as
\begin{equation}\label{eqdif4}
\begin{array}{ccl}
\left[h(t + r_0)\sqrt{1 - \psi^2}\right]_t &\leq& - h(t + r_0) \langle \nu^*,\alpha(\nu,\nu)\rangle\\
\end{array}
\end{equation}
Integrating  $(\ref{eqdif4})$ from $0$ to $t$ the resulting inequality
is the following
\[
S_{\kappa}(t + r_0)\sin\beta(\xi(t,y)) \leq S_{\kappa}(r_0)\sin\beta(y) + \int_0^t{-S_k(s + r_0)\langle\nu^*,\alpha(\nu,\nu)\rangle ds}
\]
Thus
\begin{equation}\label{eq37}
\sin\beta (\xi(t,y)) \leq\frac{h(r_0)}{h(t + r_0)}\sin\beta(y) + \frac{1}{h(t + r_0)}\int_0^t{h(s +
r_0)(-\langle\nu^*,\alpha(\nu,\nu)\rangle) ds}
\end{equation}
Since  $a(M)< 1$, then
\[
- \langle\nu^*,\alpha(\nu,\nu)\rangle(\xi(s,y)) \leq \Vert\alpha(\xi(s,y))\Vert \leq c \frac{h'}{h}(\rho_{M}(\xi(s,y))) \leq  c
\frac{h'}{h}(\rho_{N}(\xi(s,y))) 
\]But $\displaystyle\frac{h'}{h}(\rho_{N}(\xi(s,y))) = \frac{h'}{h}(s + r_0)$
for every $s \geq 0$. Substituting in (\ref{eq37}), we have
\begin{equation}\label{eq39}
\begin{array}{ccl}
\sin\beta(\xi(t,y))&\leq&\displaystyle\frac{h(r_0)}{h(t + r_0)}\sin\beta(y) + \frac{c}{h(t + r_0)}\int_0^t h'(s +
r_0)ds\\
\\
&=&\displaystyle\frac{h(r_0)}{h(t + r_0)}\sin\beta(y) + \displaystyle\frac {c}{h(t + r_0)}(h(t + r_0)-
h(r_0))\\
\\
&=&\displaystyle\frac{h(r_0)}{h(t + r_0)}(\sin\beta(y) - c) +c \\
\end{array}
\end{equation}
We will show that $\displaystyle\frac{h(r_0)}{h(t + r_0)}(\sin\beta(y) - c) +c  < 1$.
Let $\Upsilon(t)=\displaystyle\frac{h(r_0)}{h(t + r_0)}(\sin\beta(y) - c) +c$. We have that $\Upsilon(0)=\sin \beta <1$ and $\Upsilon'(t)=\displaystyle -\frac{h'(t+r_{0})h(r_{0})}{h^{2}(t+r_{0})}(\sin \beta -c)$. If $\sin \beta \geq c$ then $\Upsilon'(t)\leq 0$ and $\Upsilon(t)\leq \Upsilon(0)$. If $\sin \beta <c$, suppose by contradiction that there exists a $T>0$ such that $\Upsilon(T)>1$. This implies that $0>h(r_{0})(\sin\beta -c)> (1-c)h(T+r_{0})>0$. Then
\begin{eqnarray*}
\sin\beta(\xi(t,y))\leq \displaystyle\frac{h(r_0)}{h(t + r_0)}(\sin\beta(y) - c) +c  < 1
\end{eqnarray*}
for all $t \geq 0$. Therefore, along the integral curve $t
\mapsto\xi(t,y)$, there are no critical point for the function $R(x)
= \rho_{N} (\varphi(x))$ outside the geodesic ball $B_{N}(r_{0})$. The flow  $\xi_{t}$ maps $S_{N}(r_{0})$ diffeomorphically into to $S_{N}(r_{0}+t)$, for all $t\geq 0$. The manifold  $M$ has therefore finite topology, see also \cite{cheeger}. This
concludes the proof of the theorem \ref{logan}.
\end{proof}

Actually, the above theorem is a consequence of the convexity of the extrinsic distance  function on $M\setminus D_{r_0}$, $D_{r_0}=\varphi^{-1}(B_{N}(r_0))$, see approach given in \cite{GPGap}. In particular,

\begin{theorem}\cite{Pac2, Pac, GPGap}\label{tamed-theorem}
Let $\varphi: M \hookrightarrow N$ be an immersion  of a complete Riemannian
$m$-manifold $M$ into an  $n$-dimensional ambient manifold $N$ with a pole and radial sectional curvatures$K_N$ bounded from above by
$$
K_N\leq \kappa\leq 0.
$$
Suppose that $\varphi$ has tamed second fundamental form, then:
\begin{enumerate}
\item $\varphi$ is proper.
\item $M$ has finite topology.
\item There exists $r_0\in M$ such that the extrinsic distance function has no critical points in $M\setminus D_{r_0}$.
\item In particular, $M\setminus D_{r_0}$ is a disjoint union  $\cup_k V_k$ of finite number of ends. $M$ has so many ends $\E(M)$ as components $\partial D_{r_0}$ has , and each end $V_k$ is diffeomorphic to $\partial D_{r_0}^k\times [0,\infty)$, where $\partial D_{r_0}^k$ denotes the component of $\partial D_{r_0}$ which belongs to $V_k$.  
\end{enumerate}
\end{theorem}
 We will need the following technical lemma due to Kasue, \cite{KS}.
 
\begin{lemma}\label{kasue}{{\cite[see proof of lemma 4]{KS}}}Let $\varphi: M \hookrightarrow N$ be an immersion  of a complete Riemannian
$m$-manifold $M$ into an  $n$-dimensional ambient manifold $N$ with a pole and radial sectional curvatures$K_N$ bounded from above by
$$
K_N\leq \kappa\leq 0.
$$
Suppose that there exists a function $k:\erre\to\erre$ such that $\Vert \alpha \Vert (x)\leq k(R(x))$, $R(x)=\rho_{N}\circ \varphi (x)$, then for any $x\in M\setminus D_{r}$ with $r>r_0$,
\begin{equation}
 \vert \nabla^\perp \rho_{\kan}\vert \leq \delta(R(x))+ \frac{1}{S_{\kappa}(R(x))}\int_{r}^{R(x)}S_{\kappa}(s)k(s)ds.
\end{equation}
The function $\delta(t)$ being a decreasing function such that $\delta\to 0$ when $t\to\infty$ and  $S_\kappa$ being the solution of the following Cauchy problem
\begin{equation}
\left\{
\begin{array}{rll}
\begin{array}{l}
   S_\kappa''(t)+\kappa S_\kappa(t)=0, \\[0.1cm]
   S_\kappa(0)=0, \,\,\,\,\; S_\kappa'(0)=1.
  \end{array}
\end{array}\right.
\end{equation}

\end{lemma}

\subsection{Tamed surfaces and their topology}

This paper is concerned with tamed surfaces, hence, by Theorem \ref{tamed-theorem},  with surfaces of finite topological type.
Recall that a surface  $M$ is of finite topological type if $M$ is homeomorphic to the interior of a compact surface $\tilde M$ with non-empty boundary. A surface of finite topological type has finitely many ends. Recall also that given a compact subset $D\subset M$ of $M$, an \emph{end} $E$ of $M$ with respect to $D$ is a connected unbounded component of $M\setminus D$

 Observe that if $D_1\subset D_2$ are compact subsets of $M$, then the number of ends with respect to $D_1$ is at most the number of ends with respect to $D_2$ . This monotonicity property allows us to define the number of ends of a surface.

\begin{definition}\label{def-ends}A surface $M$ is said to have \emph{finitely many ends} if there exists $0<k<\infty$, such that, for any compact $D\subset M$, the number of  ends with respect to $D$ is at most $k$. In this case, we denote $\E(M)$ to be the smallest such $k$, and we shall say that $M$ has $\E(M)$ ends.
\end{definition}

Obviously, $\E(M)$ must be an integer and, if a surface has finitely many ends, one readily concludes that there exists $D_0\subset M$ such that, the number of ends with respect to $D_0$ is precisely $\E(M)$.  

For surfaces of finite topological type one can state the following proposition
\begin{proposition}
Suppose that $M$ is a surface of finite topological type, then:
\begin{enumerate}
\item $M$ has finitely many ends, say $\E(M)$ ends.
\item $M$ is homeomorphic to a compact surface $\widetilde M$ with $\E(M)$ points removed, \emph{i.e.},
$$
M\sim \widetilde M\setminus\left\{ p_1,\cdots,p_{\E(M)}\right\}
$$
\item There exists a compact domain $\Omega_0\subset M$, such that $M$ has $\E(M)$ ends with respect $\Omega_0$, and, every of such ends is homeomorphic to $\erre_+\times \mathbb{S}^1$ $($every end with respect to $\Omega_0$ is an annular end$)$.
\end{enumerate}
\end{proposition}

\section{Proof of theorem \ref{perimeter-th}}
Theorem \ref{perimeter-th} will be proved in two steps. In the first step we will prove theorem \ref{theo3.1} which is a version of theorem \ref{perimeter-th} but using the linear extrinsic perimeter growth property instead of the quadratic extrinsic area growth property. In the second step we will prove proposition \ref{prop3.2} where the equivalence between quadratic extrinsic area growth and linear extrinsic perimeter growth will be stated.
\begin{theorem}\label{theo3.1}Let $M$ be an immersed complete oriented  surface of $\mathbb{R}^{n}$ with curvature function $K$ and tamed second fundamental form. Suppose that $M$  admits total curvature. Then, $M$ has finite total curvature ($\int_MKdA>-\infty$) if and only if
$M$ has linear perimeter growth, \emph{i.e.}, there exists a constant $\widetilde C_1$ such that,
\begin{equation}\label{upper-theo2}
\text{L}(M\cap S_r(0))\leq \widetilde C_1 r,
\end{equation}
for $r$ large enough, where ${\rm L}(M\cap S_r(0))$ is the perimeter of the intersection of the $r$-geodesic sphere in $\erre^n$ centered at $0\in \erre^n$ with the surface $M$. Furthermore, if (\ref{upper-theo2}) holds, then there exists a constant $\widetilde C_0>0$ such that
\begin{equation}
\text{L}(M\cap S_r(0))\geq \widetilde C_0 r,
\end{equation}
for $r$ large enough.
\end{theorem}
\begin{proof}
We are going to apply the Gauss-Bonnet theorem to the extrinsic annulus $A_{r_0,t}:=D_t\setminus D_{r_0}$ for $t>r_0$. Taking into account that since the extrinsic distance function $R=\rho_{\erre^n}\circ \varphi$ has no critical points on $M\setminus D_{r_0}$, then ${A_{r_0,t}}$ is a finite union of annuli, and we obtain
\begin{equation}
\int_{A_{r_0,t}}K dA+\int_{\partial A_{r_0,t}}k_g dL=2\pi \chi({A_{r_0,t}})=0.
\end{equation} 
where $K$, $k_g$ and $\chi({A_{r_0,t}})$ denote the Gaussian curvature, the geodesic curvature and the Euler characteristic respectively.
Observe moreover that $\partial A_{r_0,t}$ is the union of two level sets
$$
\partial A_{r_0,t}=\partial D_t\cup \partial D_{r_0}.
$$
Hence,
\begin{equation}\label{Gauss-Bonn}
\int_{A_{r_0,t}}K dA=\int_{\partial D_{r_0}}k_g dL-\int_{\partial D_{t}}k_g dL
\end{equation} But for any $s$, the geodesic curvature $k_g^s$ of the extrinsic spheres $\partial D_s$ is given by
\begin{equation}
\begin{aligned}
k_g^s=&-\langle \nabla_e e,\frac{\nabla R}{\vert \nabla R\vert}\rangle=\frac{1}{\vert \nabla R\vert}\Hess_{M}R(e,e)\\
=&\frac{1}{\vert \nabla R\vert}\left(\frac{1}{s}+\langle \nabla^\perp \Rho,\alpha(e,e)\rangle\right),
\end{aligned}
\end{equation}
where $e$ is tangent to $\partial D_s$. Then,
\begin{equation}
\frac{1}{s}\frac{1}{\vert \nabla R\vert}\left(1-s\vert \nabla^\perp \Rho\vert\cdot \Vert\alpha\Vert\right)\leq k_g^s\leq \frac{1}{s}\frac{1}{\vert \nabla R\vert}\left(1+s\vert \nabla^\perp \Rho\vert \cdot\Vert\alpha\Vert\right).
\end{equation}
Since $a(M)<1$, then for any $c\in (a(M),1)$ there exists $t_c$ such that
\begin{equation}
R(x)\Vert \alpha \Vert(x)<c,
\end{equation}
for all $R(x)=t>t_c$. Using lemma \ref{kasue}, we obtain
\begin{equation}\label{ine3.6} 
\vert \nabla^\perp \Rho\vert\leq \delta(t)+\frac{c(t-t_c)}{t}\leq \delta(t)+c,
\end{equation}
at any point $x\in M$ with $R(x)=t$ and $t>t_c>r_0$. In order to simplify the notation let us denote by 
\begin{equation}\label{deflam}
\Lambda_c(t):=\delta(t)+c.
\end{equation}
Therefore,

\begin{equation}\begin{array}{lllll}
\displaystyle\frac{1}{t}\displaystyle\frac{1}{\vert \nabla R\vert}\left(1-c\vert \nabla^\perp \Rho\vert\right)&\leq & k_g^t&\leq& \displaystyle\frac{1}{t}\displaystyle\frac{1}{\vert \nabla R\vert}\left(1+c\vert \nabla^\perp \Rho\vert \right) \nonumber 
\end{array}
\end{equation}that can be simplified to
\begin{equation}\begin{array}{lllll}
\displaystyle\frac{1}{t}\displaystyle\frac{1-c\vert \nabla^\perp \Rho\vert}{\left(1-\vert \nabla^\perp \Rho\vert^2\right)^{\frac{1}{2}}}&\leq & k_g^t&\leq& \displaystyle\frac{1}{t}\displaystyle\frac{1+c\vert \nabla^\perp \Rho\vert}{\left(1-\vert \nabla^\perp \Rho\vert^2\right)^{\frac{1}{2}}}\nonumber
\end{array}
\end{equation}and can be rewritten as
\begin{equation}\label{eq3.8}
\begin{array}{lllll}
\displaystyle\frac{1-c\Lambda_c(t)}{t}&\leq & k_g^t & \leq &\displaystyle \frac{1}{t}\displaystyle\frac{1+c\Lambda_c(t)}{\left(1-\Lambda^2_c(t)\right)^{\frac{1}{2}}}
\end{array}
\end{equation}

Applying the above inequalities to the extrinsic annulus $A_{t_1,t_2}$ and using Gauss-Bonnet formula, as in the inequality (\ref{Gauss-Bonn}) we have,
\begin{equation}\label{eq3.9}\begin{array}{lll}\displaystyle\frac{{\rm L}(\partial D_{t_1})}{t_1}\displaystyle\left(1-c\Lambda_c(t_1)\right)-\displaystyle\frac{{\rm L}(\partial D_{t_2})}{t_2}\frac{1+c\Lambda_c(t_2)}{\left(1-\Lambda^2_c(t_2)\right)^{\frac{1}{2}}}&\leq & \int_{A_{t_1,t_2}}K dA
\end{array}
\end{equation}
and 
\begin{equation}\label{eq3.10}\begin{array}{lll}
\int_{A_{t_1,t_2}}K dA&\leq& \displaystyle\frac{{\rm L}(\partial D_{t_1})}{t_1}\displaystyle\frac{1+c\Lambda_c(t_1)}{\left(1-\Lambda^2_c(t_1)\right)^{\frac{1}{2}}}-\displaystyle\frac{{\rm L}(\partial D_{t_2})}{t_2}\displaystyle\left(1-c\Lambda_c(t_2)\right)
\end{array}
\end{equation}

If we suppose that $M$ has  linear extrinsic perimeter growth, from inequality (\ref{eq3.9}) and the monotonicity of $\Lambda_c$
\begin{equation}\begin{array}{lll}-\displaystyle C_1\frac{1+c\Lambda_c(t_1)}{\left(1-\Lambda^2_c(t_1)\right)^{\frac{1}{2}}}&\leq & \int_{A_{t_1,t_2}}K dA.
\end{array}
\end{equation}Letting $t_2\to\infty$, we get the desired $\int_MKdA>-\infty$ because the integral of the curvature is finite on each end of the surface.

On the other hand from inequalities (\ref{eq3.9}), (\ref{eq3.10}) and the monotonicity of $\Lambda_c$,
\begin{equation}
\begin{aligned}
\frac{{\rm L}(\partial D_{t_2})}{t_2}\leq &\frac{1}{1-c\Lambda_c(t_1)}\left[\frac{{\rm L}(\partial D_{t_1})}{t_1}\displaystyle\frac{1+c\Lambda_c(t_1)}{\left(1-\Lambda^2_c(t_1)\right)^{\frac{1}{2}}}-\int_{A_{t_1,t_2}}K dA\right],\\
\frac{{\rm L}(\partial D_{t_2})}{t_2}\geq &\frac{(1-\Lambda_c^2(t_1))^\frac{1}{2}}{1+c\Lambda_c(t_1)}\left[\frac{{\rm L}(\partial D_{t_1})}{t_1}(1-c\Lambda_c(t_1))-\int_{A_{t_1,t_2}}K dA\right].
\end{aligned}
\end{equation}

If we assume that $M$ admits finite total curvature, for any $\epsilon>0$ there therefore exists $t_1$ large enough such that
\begin{equation}
 \left\vert\int_{A_{t_1,t_2}}K dA\right\vert<\epsilon.
\end{equation}
Then
\begin{equation}
\begin{aligned}
\frac{{\rm L}(\partial D_{t_2})}{t_2}\leq &\frac{1}{1-c\Lambda_c(t_1)}\left[\frac{{\rm L}(\partial D_{t_1})}{t_1}\displaystyle\frac{1+c\Lambda_c(t_1)}{\left(1-\Lambda^2_c(t_1)\right)^{\frac{1}{2}}}+\epsilon\right]:=\widetilde C_1,\\
\frac{{\rm L}(\partial D_{t_2})}{t_2}\geq &\frac{(1-\Lambda_c^2(t_1))^\frac{1}{2}}{1+c\Lambda_c(t_1)}\left[\frac{{\rm L}(\partial D_{t_1})}{t_1}(1-c\Lambda_c(t_1))-\epsilon\right]:=\widetilde C_0.
\end{aligned}
\end{equation}
And this finishes the proof of the theorem because for $t$ large enough
\begin{equation}
\begin{aligned}
\left\vert\int_MKdA\right\vert<\infty \quad &\Longleftrightarrow \quad \exists C_1\,\,:\,\, L(\partial D_t)\leq C_1 t,\\
\left\vert\int_MKdA\right\vert<\infty \quad &\Longrightarrow \quad \exists C_0\,\,:\,\, L(\partial D_t)\geq C_0 t.
\end{aligned}
\end{equation}\end{proof}

\begin{proposition}\label{prop3.2}
Let $M$ be an immersed complete oriented  surface of $\mathbb{R}^{n}$ with tamed second fundamental form, then $M$ has quadratic extrinsic area growth, if and only if, $M$ has linear extrinsic perimeter growth. Namely,
$$
\text{L}(M\cap S_r(0))\leq \widetilde  C_1 r\Longleftrightarrow\text{A}(M\cap B_r(0))\leq C_1r^2,
$$ 
for $r$ large enough. Furthermore, 
$$
\text{L}(M\cap S_r(0))\geq \widetilde C_0 r\Longrightarrow\text{A}(M\cap B_r(0))\geq C_0r^2,
$$ 
\end{proposition}
\begin{proof}
Denote by $D_t(o)=M\cap B_t(o)$ the extrinsic ball centered at $o\in M$. Let us observe that by using coarea formula (see for instance \cite{Sakai}) for the extrinsic distance function $R=\rho_{\erre^n}\circ \varphi$ on any extrinsic ball $D_t$ with $t>t_1>r_o$, 
\begin{equation}\label{coarea1}
\begin{aligned}
{\rm A}(D_{t})=&{\rm A}(D_{t_1})+\int_{t_1}^{t}\int_{\partial D_s(o)}\frac{1}{\vert \nabla R\vert}dLds\\
\end{aligned}
\end{equation}
Thus, for any $c\in (a(M),1)$ and $t_1$ large enough, taking into account the monotonocity of the function $\Lambda_c$,
\begin{equation}
\begin{aligned}
{\rm A}(D_{t})=&{\rm A}(D_{t_1})+\int_{t_1}^{t}\int_{\partial D_s(o)}\frac{1}{\sqrt{1-\vert \nabla^\perp \Rho \vert^2}}dLds\\
\leq& {\rm A}(D_{t_1})+\int_{t_1}^{t}\int_{\partial D_s(o)}\frac{1}{\sqrt{1-\Lambda_c^2(s)}}dLds\\
=&{\rm A}(D_{t_1})+\int_{t_1}^{t}\frac{1}{\sqrt{1-\Lambda_c^2(s)}}{\rm L}(\partial D_s(o))ds\\
\leq &{\rm A}(D_{t_1})+\frac{1}{\sqrt{1-\Lambda_c^2(t_1)}}\int_{t_1}^{t}{\rm L}(\partial D_s(o))ds.
\end{aligned}
\end{equation}
Hence, if we suppose that $M$ has linear extrinsic perimeter growth,
\begin{equation}
\begin{aligned}
{\rm A}(D_t)\leq & {\rm A}(D_{t_1})+\frac{1}{\sqrt{1-\Lambda_c^2(t_1)}}\frac{\widetilde C_1}{2}(t^2-t_1^2)\\
=&\left[\frac{{\rm A}(D_{t_1})}{t^2}+\frac{1}{\sqrt{1-\Lambda_c^2(t_1)}}\frac{\widetilde C_1}{2}\left(1-\left(\frac{t_1}{t}\right)^2\right)\right]t^2\\
\leq & \left[\frac{{\rm A}(D_{t_1})}{{t_1}^2}+\frac{1}{\sqrt{1-\Lambda_c^2(t_1)}}\frac{\widetilde C_1}{2}\right]t^2
\end{aligned}
\end{equation}
Denoting $C_1:=\frac{{\rm A}(D_{t_1})}{{t_1}^2}+\frac{1}{\sqrt{1-\Lambda_c^2(t_1)}}\frac{\widetilde C_1}{2}$ we conclude that $M$ has quadratic extrinsic area growth. 

In order to prove the reverse implication let us consider now  the Laplacian of the extrinsic distance function $R=\rho_{\erre^n}\circ \varphi$,
\begin{equation}
\begin{aligned}
\triangle_MR^2=&4R\left(\frac{1}{R}+\langle \nabla \rho_{\erre^n},H\rangle\right)\leq 4\left(1+R\vert H\vert\right)\leq 4\left(1+R\Vert \alpha \Vert\right) 
\end{aligned}
\end{equation}
Applying the divergence theorem in an extrinsic ball $D_t$ with $t$ large enough and $c\in (a(M),1)$ we have 
\begin{equation}
\begin{aligned}
2t\int_{\partial D_t}\vert \nabla R\vert dL=&\int_{D_t}\triangle_MR^2\, dA=\int_{D_{t_1}}\triangle_MR^2\, dA +\int_{A_{t_1,t}}\triangle_MR^2\, dA\\
\leq&  \int_{D_{t_1}}\triangle_MR^2\, dA +4(1+c){\rm A}(A_{t_1,t})
\end{aligned}
\end{equation}
Then, denoting $A_1:=\int_{D_{t_1}}\triangle_MR^2\, dA$, and assuming that $M$ has quadratic extrinsic area growth
\begin{equation}
\begin{aligned}
2t\sqrt{1-\Lambda_c^2(t)}{\rm L}(\partial D_t)\leq& A_1+4(1+c){\rm A}(A_{t_1,t})\\
\leq & A_1+4(1+c){\rm A}(D_{t})\leq A_1+4(1+c)C_1t^2\\
\leq & \left[\frac{A_1}{t_1^2}+4(1+c)C_1\right]t^2
\end{aligned}
\end{equation}
Letting $\widetilde C_1:=\frac{\frac{A_1}{t_1^2}+4(1+c)C_1}{2\sqrt{1-\Lambda_c^2(t_1)}}$, we therefore obtain
\begin{equation}
{\rm L}(\partial D_t)\leq \widetilde C_1 t.
\end{equation}
Observe finally that from inequality (\ref{coarea1}) for any $\delta\in (0,1)$ and any $t\geq\frac{1}{1-\delta}t_1$, under the hypothesis of a lower bound for the extrinsic perimeter growth 
\begin{equation}
{\rm A}(D_t)\geq A_1+\int_{t_1}^t{\rm L}(\partial D_s)ds\geq \frac{\widetilde C_0}{2}\left(t^2-t_1^2\right)\geq \frac{\widetilde C_0}{2}t\left(t-t_1\right)\geq\frac{\delta\widetilde C_0}{2}t^2.
\end{equation}
Letting $C_0$ be $\frac{\delta\widetilde C_0}{2}$, the proposition follows.
\end{proof}

\section{Proof of theorem \ref{theo2} and corollaries \ref{nonnegative} and \ref{cor1.4}}
\begin{proof}Given a surface of finite topological type which admitting total curvature we can make use of  \cite[theorem A]{Shiohama}, for any fixed point $o\in M$
\begin{equation}
\lim_{t\to\infty}\frac{{\rm A}(t)}{t^2/2}=2\pi\chi(M)-\int_MKdA
\end{equation}
where ${\rm A}(t)$ is the area of the geodesic ball of radius $t$ centered at $o\in M$. 
Denote by $D_t(o)=M\cap B_t(o)$ the extrinsic ball centered at $o\in M$. Therefore
\begin{equation}
{\rm A}(t)\leq {\rm A}(D_t(o)).
\end{equation} Hence,
\begin{equation}
\begin{aligned}
2\pi\chi(M)-\int_MKdA=&\lim_{t\to\infty}\frac{A(t)}{t^2/2}\\
\leq & \limsup_{t\to\infty}\frac{{\rm A}(D_t(o))}{t^2/2}\\
\leq &  C_1.
\end{aligned}
\end{equation}
 The upper bound for the inequality of the theorem therefore follows. On the other hand, using the Gauss-Bonnet theorem for an extrinsic ball of radius $t$ large enough, and inequality (\ref{eq3.8}) we obtain

\begin{equation}\begin{aligned}
2\pi\chi(M)-\int_{D_t}KdA=&\int_{\partial D_t}k_gdL\geq \left(1-c\Lambda_c(t)\right)\frac{{\rm L}(\partial D_t)}{t} \\
\geq &\left(1-c\Lambda_c(t)\right)C_0.\end{aligned}
\end{equation}
Letting $t$ tend to infinity and after letting $c$ tend to $a(M)$ the theorem follows.

In order to prove corollary \ref{nonnegative}, observe that if we assume that $M$ is flat, by using the above inequality
\begin{equation}
\chi(M)=2-2g(M)-\E(M)>0
\end{equation}
where $g(M)$ is the genus of $M$ and $\E(M)$ is the number of ends of $M$. Since $\E(M)\geq 1$, the only option is $\E(M)=1$ and $g(M)=0$. The surface $M$ is therefore homeomorphic to a sphere with one point removed. Since the surface is simply connected, metrically complete and with  zero curvature, the surface is isometric to $\erre^2$ with the canonical flat metric (see \cite[theorem 11.12]{John} for instance).  

Moreover, if $M$ has strongly tamed second fundamental form, then $M$ has  tamed fundamental form as well. Hence by applying co-area formula and taking into account that $\Vert \alpha(x)\Vert\leq \frac{c}{t^{1+\epsilon}}$ for any $c\in (a(M),1)$ and $t=R(x)$ large enough,
\begin{equation}
\begin{aligned}
\int_M\Vert \alpha\Vert^2 dA=&\int_{D_{t_1}}\Vert \alpha\Vert^2 dA+\int_{A_{t_1,t}}\Vert \alpha\Vert^2 dA\\
\leq & \int_{D_{t_1}}\Vert \alpha\Vert^2 dA+\int_{t_1}^t\int_{\partial D_s}\frac{\Vert \alpha\Vert^2}{\vert \nabla R\vert} dLds\\
\leq & \int_{D_{t_1}}\Vert \alpha\Vert^2 dA+\int_{t_1}^t\frac{c^2}{s^{2+2\epsilon}\sqrt{1-\Lambda_c^2(s)}}L(\partial D_s)ds\\
\leq & \int_{D_{t_1}}\Vert \alpha\Vert^2 dA+\frac{c^2\, \widetilde C_1}{\sqrt{1-\Lambda_c^2(t_1)}}\int_{t_1}^t\frac{1}{s^{1+2\epsilon}}ds<\infty.
\end{aligned}
\end{equation}
By using now theorem 2 of \cite{W}, $\int_M KdA$ is an integral multiple of $2\pi$, and using the lower bounds given  in the inequality of  theorem \ref{theo2} we conclude that $\int_M KdA=0$ because $\chi(M)=1$. Since $M$ is a complete and flat surface with  tamed second fundamental form, $M$ is therefore isometric to $\erre^2$ and this finishes the proof of corollary \ref{cor1.4}
\end{proof}
\section{ Proof of Theorem \ref{tone-theo}}

 The first ingredient for the proof of Theorem
\ref{tone-theo} is  Barta's Theorem \cite{barta}.

\begin{theorem}[Barta]Let $\Omega$ be a bounded
 open set with piecewise  smooth
boundary  in  a Riemannian manifold.  Let $f \in C^2(\Omega)\cap C^0(\bar \Omega)$ with
$f|\Omega > 0$ and $f|\partial \Omega =0$. Then the first Dirichlet
eigenvalue $\lambda_{1}(\Omega)$ has the following bounds:
\begin{equation}\label{eqBarta}
\begin{array}{ccl}
\displaystyle\sup_{\Omega} (- \frac {\Delta f}{f}) \geq \lambda_1(\Omega)
\geq  \displaystyle\inf_{\Omega}(-\frac{\Delta f}{f})\\
\end{array}
\end{equation}
With equality in \eqref{eqBarta} if and only in $f$ is a positive first eigenfunction
of $\Omega$.\end{theorem} We now present the proof of Theorem \ref{tone-theo}.
\begin{proof}

 \noindent Let $\varphi: M \hookrightarrow  N$ be an isometric
immersion with tamed second fundamental form of a complete $m$-manifold $M$ into a $n$-manifold $N$   with a pole $p\in N$ and sectional radial curvatures $B\leq K_{N}\leq 0$.
Let $x_{0}\in M$, $p=\varphi(x_{0})\in N$ and let $\rho_{N} (y)={\rm dist}_{N}(p,y)$ be the
distance function on $N$ and $\rho_{N} \circ \varphi$ the extrinsic distance on $M$. By the proof of Theorem (\ref{logan}) there is an $r_{0}>0$ such that there is no critical points $x\in
M\setminus \varphi^{-1}(B_{N}(r_{0}))$ for $\rho_{N} \circ \varphi$, where
$B_{N}(r_{0})$ is the geodesic ball in $N$ centered at $p$ with radius $r_{0}$. Let $r >r_{0}$ and let $D_r =\varphi^{-1}(B_{N}(r))$ be an extrinsic ball. Since $\varphi $
is proper we have that $D_t$ is precompact with boundary $\partial D_t$ that we may suppose to be smooth for any $t>r_0$ by using the regular set theorem. Let $v:B(r)\to \mathbb{R}$ be a positive first eigenfunction of the geodesic ball $B(r)$ of radius $r$ in the $l$-dimensional Euclidean space $\erre^l$, where $l$ is to be determined. The function $v$ is radial, i.e. $v(x)=v(\vert x\vert)$, and satisfies the following differential equation,
\begin{equation}v''(t)+(l-1)\,\frac{v'(t)}{t}\,
+\lambda_{1}(r)v(t)=0,
\,\, \forall\, t\in[0,r].\label{eqLambda-l}
\end{equation}
With initial data $v(0)=1$, $v'(0)=0$. Moreover, $v'(t)<0$ for all $t\in (0,r]$. Where $\lambda_{1}(r)$ is the first Dirichlet eigenvalue of the geodesic ball $B(r)\subset\erre^l$ with radius $r$.
Define
$\tilde{v}:B_{N}(r)\to \mathbb{R}$ by $\tilde{v}(y)=v\circ\rho_{N}(y)$
 and $f:D_r \to \mathbb{R}$ by $f(x)=\tilde{v}\circ \varphi(x)$. By  Barta's Theorem we have
 $\lambda_{1}(D_r)\leq \sup_{D_r}(-\triangle f/f)$. The Laplacian $\triangle f$ at a point $x\in M$ is given by
 \begin{eqnarray}
 \triangle_{M} f(x)& =&[\displaystyle\sum_{i=1}^{m}\hess\, \tilde{v}(e_i,e_i)
 + \langle \grad \tilde{v},\vec H\rangle](\varphi(x))\nonumber \\
 &=&\displaystyle\sum_{i=1}^{m}\left[{v}''(\rho_{N}
 )\langle \grad\rho_{N}, e_{i}\rangle^{2} +v'( \rho_{N}
 )\,\hess\,\rho_{N} (e_i,e_i)\right]+ v'( \rho
 )\langle \grad \rho_{N},\vec H\rangle\nonumber
 \end{eqnarray}
 Where $\hess\,\tilde{v}$ is the Hessian of $\tilde{v}$ in the metric
 of $N$ and $\{e_{i}\}_{i=1}^{m}$ is an orthonormal basis for $T_{x}M$ where
 we made the identification $d\varphi(e_{i})=e_{i}$. We are going to give an upper bound for $(-\triangle f/f)$ on $\varphi^{-1}(B_{N}(r))$.
Let $x\in \varphi^{-1}(B_{N}(r)) $ and choose an orthonormal basis $\{e_1,...,e_m\} $ for $T_{x}M$ such that $\{e_{2},\ldots, e_{m}\}$ are tangent to the distance sphere $\partial B_{N}(R(x))$ of radius $R(x)=\rho_{N} (\varphi (x))$ and $e_{1}=\frac{\nabla R}{\vert \nabla R\vert}$. To simplify the notation set $t=\rho_{N}(\varphi(x))$, $\triangle_{M}=\triangle$. Then
\begin{eqnarray}\label{eq40}
\triangle f(x)&=&\displaystyle\sum_{i=1}^{m}\left[{v}''(t)\langle \grad \rho_{N}, e_{i}\rangle^{2} +v'(t)\,\hess\,\rho_{N} (e_i,e_i)\right]+ v'( t)\langle \grad \rho_{N},\vec H\rangle\nonumber \\
&=&{v}''(t)\langle \grad \rho_{N}, \frac{\nabla R}{\vert \nabla R\vert}\rangle^{2}+v'(t)\,\hess\,\rho_{N} (\frac{\nabla R}{\vert \nabla R\vert},\frac{\nabla R}{\vert \nabla R\vert}) \\
&&+\displaystyle\sum_{i=2}^{m}v'(t)\,\hess\,\rho_{N} (e_i,e_i)+ v'( t)\langle \grad\rho_{N},\vec H\rangle\nonumber
\end{eqnarray}

Thus from (\ref{eq40})
\begin{eqnarray}\label{eq41}-\frac{\triangle f}{f}(x)&=&-\frac{v''}{v}(t)\langle \grad \rho_{N}, \frac{\nabla R}{\vert \nabla R\vert}\rangle^{2}-\frac{v'}{v}(t)\,\hess\,\rho_{N} (\frac{\nabla R}{\vert \nabla R\vert},\frac{\nabla R}{\vert \nabla R\vert}) \\
&&-\displaystyle\sum_{i=2}^{m}\frac{v'}{v}(t)\,\hess\,\rho_{N} (e_i,e_i)-\frac{ v}{v}'( t)\langle \grad\rho_{N},\vec H\rangle\nonumber
\end{eqnarray}
The equation (\ref{eqLambda-l})  says that \[ -\frac{v''}{v}(t)=(l-1)\frac{v'}{t\, v}+ \lambda_{1}(r)\]
By the Hessian Comparison Theorem and the fact $v'/v\leq 0$ we have  from equation (\ref{eq41}) the following inequality

\begin{eqnarray}\label{eq42}-\frac{\triangle f}{f}(x)&\leq &\lambda_{1}(r)\left(1-\vert \nabla^\perp R\vert^2\right) \nonumber\\
&& -\frac{v'}{tv}\left[\frac{t\,h'}{h}(m-\vert \nabla^\perp R\vert^2)-(l-1)\vert \nabla R\vert^2+t\vert \vec H\vert\right]\nonumber\\
&\leq &\lambda_{1}(r)\\
&& -\frac{v'}{tv}\left[\frac{t\,h'}{h}m-(l-1)\vert \nabla R\vert^2+t\vert \vec H\vert\right]\nonumber
\end{eqnarray}
where $h$ is the solution of the following problem 
\begin{eqnarray}\label{eqh}
\left \{
\begin{array}{l}
h'' +Bh=0 \\
h(0)=0, h'(0)=1\\
\end{array} \right.
\end{eqnarray}
Now, to bound $t\frac{h'}{h}$ we will make use of the following lemma
\begin{lemma}
Let $h\in C^\infty [0,\infty)$ be a positive function with $h(0)=0$ and $h'(0)=1$. Suppose 
$$
 \frac{h''}{h}(t)\leq \frac{2}{t^2},\quad \forall t>0.
$$
Then
$$
t\frac{h'}{h}(t)\leq 2.
$$
\end{lemma}
\begin{proof}
Observe that the function $h'(t)t^2-2h(t)t$ is a decreasing function on $t$ because
\begin{equation}
\frac{d}{dt}\left(h'(t)t^2-2h(t)t\right)=h''(t)t^2-2h(t)=h(t)t^2\left(\frac{h''(t)}{h(t)}-\frac{2}{t^2}\right)\leq 0.
\end{equation}
Hence for any $t>0$
\begin{equation}
h'(t)t^2-2h(t)t\leq h'(t_0)t_0^2-2h(t_0)t_0\leq h'(t_0)t_0^2,
\end{equation}
for any $t_0<t$. Then,
\begin{equation}
\frac{h'(t)t}{h(t)}\leq \frac{h'(t_0)t_0^2}{h(t)t}+2,
\end{equation}
Now letting $t_0$ tend to $0$ we obtain the desired upper bound.
\end{proof}By using the above lemma in inequality (\ref{eq42}),
\begin{eqnarray}\label{eqfinal}-\frac{\triangle f}{f}(x)&\leq &\lambda_{1}(r)-\frac{v'}{tv}\left[2m-(l-1)\vert \nabla R\vert^2+t\vert \vec H\vert\right]
\end{eqnarray}

Since the immersion is tamed we have that there exists $t_c$ such that for any $c\in (a(M),1)$
\begin{equation}
R(x)\Vert \alpha\Vert(x)\leq c,\quad\forall x\in M\setminus D_{t_c}.
\end{equation} 
We are going to split the prove in two cases

\subsection*{Case I} The point $x\in D_r$ belongs to $M\setminus D_{t_c}$

Since we are assuming that $x\in M\setminus D_{t_c}$, then by using inequality (\ref{eqfinal})

\begin{equation}
\begin{aligned}
-\frac{\triangle f}{f}(x) \leq &\lambda_{1}(r) -\frac{v'}{tv}\left[2m-(l-1)(1-\Lambda_c(t_c)^2)+c\right].
\end{aligned}
\end{equation}
Because $t\vert\vec H\vert\leq t\vert \alpha\vert\leq c$ and we have used the monotonocity of the $\Lambda_c$ function given in definition (\ref{deflam}), see also inequality (\ref{ine3.6}).
Since the above inequality is true for any $\mathbb{N}\ni l\geq 1$, we can choose $l$ large enough in such a way that 
$$2m-(l-1)(1-\Lambda_c(t_c)^2)+c\leq 0.$$
 Hence,
\begin{equation}
-\frac{\triangle f}{f}(x) \leq \lambda_{1}(r),
\end{equation} 
for any $x\in M\setminus D_{t_c}$.

\subsection*{Case II}The point $x\in D_t$ belongs to $D_{t_c}$.

Since $D_{t_c}$ is compact, let us set
\begin{equation}
H_0:=\max_{x\in D_{t_c}}R(x)\vert \vec H\vert.
\end{equation}
By using inequality (\ref{eqfinal}),
\begin{eqnarray}\label{eqfinal2}-\frac{\triangle f}{f}(x)&\leq &\lambda_{1}(r)-\frac{v'}{tv}\left[2m+H_0\right]
\end{eqnarray} 
 We need the following technical lemma.
 \begin{lemma}\label{lemma2}
 Let $v$ be the function satisfying (\ref{eqLambda-l}).   Then,  
$$
-v'(t)/t \leq \lambda_{1}(r)
$$ for all $t\in [0,r]$.
 \end{lemma}
\begin{proof}Consider the function $\gamma:[0,r]\to \mathbb{R}$ given by
$\gamma(t)=\lambda_1(r) \cdot t + v'(t)$. We know that $v(0)=1$, $v'(0)=0$ and $v'(t)\leq 0$ besides $v$ satisfies equation (\ref{eqLambda-l}). Observe that \[0=v''(t)+(l-1)v'+\lambda_1(r) v\leq v'' +\lambda_1(r).\] Thus $v''\geq -\lambda_1(r)$ and $\gamma'(t)=\lambda_1(r) +v''\geq 0$. Since $\gamma(0)=0$ we have $\gamma(t)=\lambda_1(r) t+v'(t)\geq 0$. This proves the lemma.\end{proof}

 Since $v$ is a non-increasing positive function we have $ v(t) \geq
v(t_c)$. Applying the inequality (\ref{eqfinal2})  we obtain
\begin{eqnarray}\label{eqfinal3}-\frac{\triangle f}{f}(x)&\leq &\lambda_{1}(r)\left[1+\frac{1}{v(t_c)}\left(2m+H_0\right)\right].
\end{eqnarray} 

Thus, finally from case I and Case II, we know that  for all $x\in \varphi^{-1}(B_{N}(r))$
\begin{eqnarray}
-(\triangle f/f)(x)& \leq&
 \max \left\{ 1, 1+\frac{1}{v(t_c)}\left(2m+H_0\right)\right\}\cdot\lambda_{1}(r)\nonumber \\ &=&
\left[1+\frac{1}{v(t_c)}\left(2m+H_0\right)\right]\cdot\lambda_{1}(r)\nonumber
\end{eqnarray}
Then by Barta's Theorem \[\lambda_{1}(D_r)\leq \left[1+\frac{1}{v(t_c)}\left(2m+H_0\right)\right]\cdot\lambda_{1}(r)\] Observe that $\left[1+\frac{1}{v(t_c)}\left(2m+H_0\right)\right]$ does not depend on $r$. So letting $r\to \infty$ we have 
$$
\lambda^{\ast}(M)\leq \left[1+\frac{1}{v(t_c)}\left(2m+H_0\right)\right]\cdot\lambda^{\ast}(\erre^{l})=0.
$$
And this finishes the proof of the theorem.\end{proof}

\def\cprime{$'$} \def\cprime{$'$} \def\cprime{$'$} \def\cprime{$'$}
  \def\cprime{$'$}

\end{document}